\documentclass[12pt]{amsart}
\usepackage{color}
\usepackage[all]{xy}
\usepackage{amssymb}
\calclayout

\newtheorem{dummy}{anything}[section] 
\newtheorem{theorem}[dummy]{Theorem}
\newtheorem*{thma}{Theorem A}
\newtheorem*{thmb}{Theorem B}
\newtheorem*{cor}{Corollary}
\newtheorem{lemma}[dummy]{Lemma} 
\newtheorem{proposition}[dummy]{Proposition} 
\newtheorem{corollary}[dummy]{Corollary}
 
\theoremstyle{definition}
\newtheorem{definition}[dummy]{Definition}
 \newtheorem{example}[dummy]{Example}
 
 \newtheorem{remark}[dummy]{Remark}

\newcommand
{\eqncount}{\setcounter{equation}{\value{dummy}}%
\addtocounter{dummy}{1}}
\newcommand{\cA}{\mathcal A}
\newcommand{\cB}{\mathcal B}
\newcommand{\cC}{\mathcal C}
\newcommand{\cD}{\mathcal D}

\newcommand{\cF}{\mathcal F}

\newcommand{\cV}{\mathcal V}

\newcommand{\bZ}{\mathbf Z}

\newcommand{\Or}{\mathbf{Or}}

\newcommand{\wX}{\overline X}

\newcommand{\wD}{\overline D}
\newcommand{\wZ}{\overline Z}
\newcommand{\wY}{\overline Y}
\newcommand{\wS}{\overline S}
\newcommand{\wT}{\overline T}

\newcommand{\vv}{\, | \,}

\newcommand{\mmatrix}[4]{\left (\vcenter
{\xymatrix@C-2pc@R-2pc{#1&#2\\#3&#4} }
\right )}

\DeclareMathOperator{\Ind}{Ind}
\DeclareMathOperator{\Res}{Res}

\newcommand{\Gspaces}{\text{$\G$-$CW$-Complexes}}
\newcommand{\sst}{\scriptscriptstyle}
\newcommand{\nms}{\negmedspace}

\newcommand{\LG}{\mathbb L^{\nms\raise 1pt\hbox{$\sst geom$}}}
\newcommand{\LA}{\mathbb L^{\nms\raise 1pt\hbox{$\sst alg$}}}
\newcommand{\Ls}{\mathbb L^{\negthickspace\raise 1pt\hbox{$\sst s$}}}
\newcommand{\LI}{\mathbb L^{\nms\raise 1pt\hbox{$\sst -\infty$}}}
\newcommand{\KI}{\mathbb K^{\nms\raise 1pt\hbox{$\sst -\infty$}}}
\newcommand{\Lh}{\mathbb L^{\nms\raise 1pt\hbox{$\sst h$}}}
\newcommand{\Lp}{\mathbb L^{\negthickspace\raise 1pt\hbox{$\sst p$}}}
\newcommand{\bLI}{\mathbf L^{\nms\raise 1pt\hbox{$\sst-\infty$}}}
\newcommand{\nLI}{L^{\nms\raise 1pt\hbox{$\sst-\infty$}}}

\newcommand{\G}{G}

\newcommand{\Vertical}{\,|\,}
\newcommand{\BRcat}[1]{\cB_{\G}(#1;R)}

\newcommand{\DAcat}[1]{\cD(#1;\cA)}
\newcommand{\DAB}[3]{\cD^{#2}(#1;#3)}
\newcommand{\DABl}[3]{\cD^{#2}_l(#1;#3)}
\newcommand{\Dcatl}[3]{\cD^{#2}_l(#1\times [0,1);#3)}

\newcommand{\DAcatG}[1]{\cD^{G}(#1\times [0,1);\cA)}
\newcommand{\DAcatH}[3]{\cD^{#2}(#1\times [0,1);#3)}
\newcommand{\Dlcat}[3]{\cD^{#2}_l(#1\times [0,1);#3)}
\newcommand{\Iopen}{[0,1)}

\newcommand{\Lbar}[1]{{\overset{\, \hrulefill\, }{#1}}}

\DeclareMathOperator{\supp}{supp}

\newcommand{\nr}[1]{\medskip\noindent{\bf (${\mathbf #1}$)}.}

\begin{document}
\title[Assembly for Group Extensions]
{Assembly Maps for Group Extensions in $K$-Theory and $L$-Theory With Twisted Coefficients}
\author{Ian Hambleton} 
\thanks{Partially supported by NSERC grant A4000
and  NSF grant  DMS 9104026. The authors also
wish to thank the SFB 478, Universit\"at M\"unster,
for hospitality and support.}
\address{Department of Mathematics \& Statistics
 \newline\indent
McMaster University
 \newline\indent
Hamilton, ON  L8S 4K1, Canada}
\email{ian{@}math.mcmaster.ca}
\author{Erik K. Pedersen}
\address{Department of Mathematical Sciences
 \newline\indent
SUNY at Binghamton
 \newline\indent 
Binghamton, NY 13901, USA}
\email{erik@math.binghamton.edu}
\author{David Rosenthal}
\address{Department of Mathematics and Computer Science
 \newline\indent
St. John's University
 \newline\indent 
Jamaica, NY 11439, USA}
\email{rosenthd@stjohns.edu}
\date{March 31, 2008} 
\begin{abstract}\noindent
In this paper we show that the Farrell-Jones isomorphism conjectures are inherited in group extensions for assembly maps in $K$-theory and $L$-theory with twisted coefficients.
\end{abstract}
\maketitle
\section*{Introduction}
Under what assumptions are the Farrell-Jones isomorphism conjectures inherited by group extensions or subgroups? We will formulate a version of the standard conjectures (see Farrell-Jones \cite{farrell-jones1}) with twisted coefficients in an additive category, and then study these questions via the continuously controlled assembly maps of \cite[\S 7]{hp4}. A formulation using the Davis-L\"uck assembly maps \cite{davis-lueck1} has already been given by Bartels and Reich \cite{bartels-reich2}, and applied there to show inheritance by subgroups. Recall that the Farrell-Jones conjecture in algebraic $K$-theory asserts that certain ``assembly" maps 
$$H^G_n(E_{\cV\cC}G; \mathbb K_R) \to K_n(RG)$$
are isomorphisms, for a given ring $R$, and all $n\in \bZ$. Here the space 
$E_{\cV\cC}G$ is the universal $G$-CW-complex for $G$-actions with virtually cyclic isotropy, and the left-hand side denotes equivariant homology with coefficients in the non-connective $K$-theory spectrum for the ring $R$.

\begin{thma}  Let $N \to G \xrightarrow{\pi} K$ be a group extension, where $N\triangleleft\, G$ is a normal subgroup, and $K$ is the quotient group. Let 
$\cA$ be an additive category with  $G$-action. Suppose that
\begin{enumerate}
\item The group $K$ satisfies the Farrell-Jones conjecture in algebraic $K$-theory, with  twisted coefficients in any additive category with $K$-action.
\item Every subgroup of $G$ containing $N$ as a subgroup, with virtually cyclic quotient, satisfies the Farrell-Jones conjecture in algebraic $K$-theory, with twisted coefficients in $\cA$.
\end{enumerate}
Then the group $G$ satisfies the Farrell-Jones conjecture in algebraic $K$-theory, with  twisted coefficients in $\cA$. 
\end{thma}
This is a special case of a more general result (see Theorem \ref{mainone}). The same statement holds for algebraic $L$-theory as well, where the coefficient categories are additive categories with involution. The corrresponding result for the Baum-Connes conjecture was obtained by Oyono-Oyono \cite{oyono-oyono1}, and our proof follows the outline given there.
One of the main points is that the most effective methods known for proving the standard Farrell-Jones conjectures (for particular groups $G$) also work for the twisted coefficient versions (compare  \cite{bartels-farrell-jones-reich1}, \cite{bartels-reich1}, \cite{cp1}, \cite{cp2}, \cite{rosenthal1}, \cite{rosenthal2}, and  \cite{rosenthal-schutz1}). An immediate corollary to Theorem A is the following.

\begin{cor}[Corollary~\ref{vproducts}]
The Farrell-Jones conjecture with twisted coefficients is true for $\G_1\times \G_2$ if and only if it is true for $\G_1$,  $\G_2$, and every product
$V_1 \times V_2$, where $V_1\leq \G_1$ and $V_2\leq \G_2$ are virtually cyclic subgroups.
\end{cor}

The fibered isomorphism conjecture of Farrell and Jones~\cite{farrell-jones1} for a group $G$ and a ring $R$ asserts that for every group homomorphism, $\phi\colon H\to G$, the assembly map for $H$ relative to the family generated by the subgroups $\phi^{-1}(V)$, $V \subset G$ virtually cyclic, is an isomorphism. This conjecture implies the Farrell-Jones conjecture and has better inheritance properties. For example, the fibered version of our Theorem A is also true (see, for example, \cite[Section 2.3]{bartels-lueck-reich1}). The following result shows that the Farrell-Jones conjecture with twisted coefficients implies the Fibered Farrell-Jones conjecture.

\begin{thmb} Suppose that $\phi\colon H\to G$ is a group homomorphism. Then the Farrell-Jones conjecture for $G$, with twisted coefficients in any $G$-category, implies that the assembly map for $H$  relative to the family generated by the subgroups $\phi^{-1}(V)$, $V \subset G$ virtually cyclic, is an isomorphism with twisted coefficients in any $H$-category.
\end{thmb}

The corresponding result for the Davis-L\"uck assembly maps was obtained by Bartels-Reich \cite{bartels-reich2}, who also pointed out a number of applications of the assembly map with twisted coefficients, including the study the $K$- and $L$-theory of twisted group rings (see also Example \ref{exone} and Example \ref{extwo} below). One can check as in \cite{hp4} that those assembly maps are equivalent to the continuously controlled assembly maps used in this paper.

\section{ Assembly via Controlled Categories}\label{}
The controlled categories of
Pedersen \cite{p19}, Carlsson-Pedersen \cite{cp1},
\cite{cpv1} are our main tool for identifying
various different assembly maps. We will recall
the definition of these categories, and then the usual assembly maps
are obtained by applying  functors 
$$H\colon \Gspaces \to Spectra$$ as described in \cite{hp4}. We will extend the earlier definitions in order to allow an additive category as coefficients, instead of just working with modules over a ring $R$. 
A formulation for assembly maps with coefficients in the setting of \cite{davis-lueck1} has already been given in \cite{bartels-reich2}. Following the method of \cite{hp4}, one can check that the two different descriptions give the same assembly maps.

Let $\G$ be any discrete group, and let $X$ be
a $\G$-CW complex (we will use a left $G$-action). Subspaces of the form $\G \cdot D \subset X$, with $D$ compact in $X$, are called $\G$-\emph{compact} subspaces of $X$. More generally, a subspace whose closure has this form is called relatively $\G$-compact.
A \emph{resolution} of $X$ is a pair $(\wX, p)$, where $\wX$ is a free $\G$-CW complex and 
$p\colon \wX \to X$ is a continuous $G$-equivariant map, such that
for every $\G$-compact set $\G\cdot D\subset X$ there exists a $G$-compact set $\G\cdot\wD \subset \wX$ such that $p(\G\cdot\wD) = \G \cdot D$. The notion of resolution comes from \cite{p19}, and was developed further in  \cite[\S 3]{bartels-farrell-jones-reich1}. The original example was $\wX =  \G\times X$, with the diagonal $\G$-action and first factor projection.

Let $\cA$ be an additive category with involution, and suppose that $\cA$ has a right $G$-action compatible with the involution. This is a collection of covariant functors $\{g^*\colon \cA \to \cA, \forall g\in G\}$, such that $(g\circ h)^* = h^*\circ g^*$ and $e^* = id$. We require
that the functors $g^*$ commute with the involution $\ast\colon
\cA \to \cA$ (an involution is a contravariant functor with square the identity).

\begin{definition} Let $(Z, X)$ be a $G$-CW pair, where $X$ is a closed $G$-invariant subspace. Let $Y = Z -X$, and fix a resolution $p\colon \wZ \to Z$, whose restriction to $Y$ is denoted $\wY$.
The  category $\DAcat{Z, X}$ has
objects $A = (A_{y})$ consisting of a collection of objects of $\cA$, indexed by $y \in \wY$, and 
morphisms  $\phi\colon A \to B$  consisting
of collections  $\phi =(\phi_y^z)$ of morphisms
$\phi_y^z\colon A_y \to B_z$ in $\cA$, indexed by $y,z \in \wY $, satisfying:
\begin{enumerate}
\item the support $\{ y \in \wY \Vertical A_y \neq 0\}$
is \emph{locally finite} in $\wY$, and relatively $\G$-compact
in $\wZ$.

\item for each morphism $\phi\colon A \to B$, and for each $y\in \wY$, the set
$\{z \vv   \phi_y^z\neq 0 \text{\ or\ } \phi_z^y\neq 0\}$  is finite.

\item the morphisms $\phi\colon A \to B$ are \emph{continuously controlled} at
$X \subset Z$. For every 
$x \in X$, and for every $\G_x$-invariant neighbourhood $U$ of 
$x$ in $Z$, there is a $\G_x$-invariant neighbourhood
$V$ of $x$ in $Z$ so that
$\phi_y^z =0$ and $\phi_z^y =0$ whenever
$p(y) \in (Y -U)$ and $p(z) \in (V \cap U\cap Y)$. 
\end{enumerate}

\end{definition}
If $X=\emptyset$, we use the shorter notation
$\DAcat{Z}:= \DAcat{Z,\emptyset}$,
and in this case the continous control condition (iii) on morphisms is vacuous.
 If $S$ is a discrete left $G$-set, we denote by
$\DABl{S\times Z,S\times X}{}{\cA}$ the subcategory where the morphisms are $S$-level-preserving: $\phi_{(s,y)}^{(s',z)} = 0$ if $s \neq s' \in S$, for any $y,z \in Y$.

The category $\DAcat{Z,X}$ is an additive category with involution, 
where the dual of $A$ is given by $(A^*)_y=A^*_y$ for all $y \in \wY $.
It depends functorially on the pair $(Z,X)$ of $\G$-CW complexes.
The actions of $G$ on $\cA$ and $Z$ induce a right $\G$-action on
$\DAcat{Z,X}$. For $g\in \G$, we set $(gA )_y= g^*A_{gy}$ and
$(g\phi)_y^z = g^*(\phi_{gy}^{gz})$. The fixed subcategory will be denoted $\DAB{Z,X}{\G}{\cA}$. If $\G=\{e\}$ is the trivial group, we use the notation $\DAB{Z,X}{0}{\cA}$.
We have not included the resolution $(\wZ, p)$ in the notation, because two different resolutions give $\G$-equivalent categories
(see \cite[Prop.~3.5]{bartels-farrell-jones-reich1}). 
We can compare these fixed subcategories to the equivariant category $\BRcat{Z,X}$ defined in \cite[\S 7]{hp4}.

\begin{lemma} There is an equivalence of categories
$\BRcat{Z,X} \simeq \DAB{Z,X}{\G}{\cA}$, when $\cA$ is the category of finitely-generated free $R$-modules.
\end{lemma}
\begin{proof} We define a functor $F\colon  \DAB{Z,X}{\G}{\cA} \to \BRcat{Z,X}$ by
sending an object $A$ to the free $R$-module $F(A)_y = \oplus_{g\in G_y} A_{(g,y)}$, for all $y \in Y$, with the obvious reference map to $Y$. Similarly, for a morphism $\phi\colon A \to B$, we define $F(\phi)_y^z = (\phi_{g,y}^{g',z})_{g, g' \in G}$, for all $y, z \in Y$. The verification that this definition makes sense will be left to the reader.

 Conversely, we can define a functor $F'\colon  \BRcat{Z,X} \to \DAB{Z,X}{\G}{\cA}$ on objects by decomposing an object $A = (A_y)$ of $\BRcat{Z,X}$ as $A_y=\oplus_{g\in G_y} (A_{y})_g$, since $A_y$ is a finitely-generated free $RG_y$-module. Now we let
$F'(A)_{(g,y)} = (A_{y})_g$, for all $y \in Y$, $g\in G$, and on morphisms by letting $F'(\phi)_{g,y}^{g',z} = \phi_{gy}^{g'z}$. Again the verifications will be left to the reader (technically we should work with a  category equivalent to $\BRcat{Z,X}$, in which the objects are based: each $A = R[T]$, where $T$ is a free $G$-set, and $T$ is equipped with a reference map to $X \times [0,1]$).
\end{proof}

For applications to assembly maps, we will let $X$ be a $\G$-CW complex and $Z = X \times [0,1]$ so that $Y= X \times [0,1)$.
The category just defined will be denoted
$$\DAcatG{X}:=\cD^G(X\times [0,1], X\times 1;\cA)\ .$$

\medskip
Let $\DAcatG{X}_{\emptyset}$ denote the full
subcategory of $\DAcatG{X}$ with objects
$A$ such that the intersection with the closure
$$ \text{supp}(A) = \Lbar{\{(x,t) \in \wX\times \Iopen \Vertical
A_{(x,t)} \neq 0\}}\cap (X\times 1)$$
is the empty set.

\begin{example}
If $\cA$ is the additive category of finitely generated free $R$-modules, then $\DAcatG{X}_{\emptyset}$ is equivalent
to the category of finitely generated free
$R\G$-modules, for any $\G$-CW complex $X$.
\end{example}

The quotient category
will be denoted $\DAcatG{X}^{>0}$, and
we remark that this is a germ category 
 (see \cite[\S 7]{hp4},  \cite{pw2}, \cite{cp1}). 
The objects
are the same as in $\DAcatG{X}$ but
morphisms are identified if they agree close
to $\wX = \wX\times 1$ (i.e. on the complement
of a neighbourhood of $\wX\times 0$).
Here is a useful remark.
\begin{lemma}[\cite{hp4}]\label{levels}  Let $S$ be a discrete left $G$-set. The forgetful functor
$$\Dcatl{S\times X}{G}{\cA}^{>0} \to 
\DAcatG{S\times X}^{>0}$$
is an equivalence of categories.
\end{lemma}
\begin{proof}
In the germ category, every morphism has a representative which is  level-preserving with respect to projection on $S$.
\end{proof}

The category $\DAcatG{X}^{>0}$ is an additive category 
with involution, and
we obtain a functor 
$\Gspaces \to \text{AddCat}^{-}$.
The results of \cite[1.28,~4.2]{cap1} now show
that the functors $F^\lambda\colon \Gspaces \to Spectra$
defined by
\eqncount
\begin{equation}\label{CCfun}
F^\lambda_\G(X;\cA) :=\begin{cases} \KI(\DAcatG{X}^{>0})\\
\LI(\DAcatG{X}^{>0})\end{cases},
\end{equation}
where $\lambda=\KI$ or $\lambda=\LI$ respectively,
are $\G$-homotopy invariant and $\G$-excisive. 

We can now extend the definition of the assembly maps to allow coefficients in any additive category with $\G$-action.
\begin{definition}
We define the
\emph{continuously controlled assembly map} \emph{with coefficients in} $\cA$ to be the map $F^\lambda_\G(X;\cA) \to F^\lambda_\G(\bullet;\cA)$.
\end{definition}

\medskip From the methods of~\cite{hp4}, the continuously controlled assembly map with coefficients is homotopy equivalent to the assembly map with coefficients constructed in~\cite{bartels-reich2}. The most important example to consider is when $X=E_{\cV\cC}G$, in which case the \emph{Farrell-Jones conjecture with coefficients} asserts that this assembly map is an equivalence. Given a discrete group $G$, a family of subgroups $\cF$ of $G$, and coefficients $\cA$, we will refer to $$ F^\lambda_\G(E_\cF G;\cA) \to F^\lambda_\G(\bullet;\cA) $$
	as the $(G,\cF,\cA)$-{\it assembly map}.

By applying $\KI$  or $\LI$ to the
sequence of additive categories (with involution):
$$ \DAcatG{X}_\emptyset \to \DAcatG{X} 
\to \DAcatG{X}^{>0}$$
we obtain a fibration of spectra \cite{cp1}. 
As in \cite{hp4}, we have
the following description for 
the assembly map.
\begin{theorem}[{\cite[\S 7]{hp4}}]
The continuously controlled
 assembly map $$F^\lambda_\G(X;\cA) \to F^\lambda_\G(\bullet;\cA)$$
  is homotopy equivalent to the connecting map
$$\lambda( \DAcatG{X}^{>0}) \to 
\Omega^{-1}\lambda( \DAcatG{X}_{\emptyset}$$
for  $\lambda=\KI$ or $\lambda=\LI$.
\end{theorem}
See \cite[\S 2]{hp4} for the definition of homotopy equivalent functors from 
$$\Gspaces \to Spectra, $$ and \cite[5.1]{davis-lueck1} for the result that any functor $E\colon\Or(G) \to Spectra$ out of the orbit category of
$\G$ may be extended uniquely (up to homotopy) to a functor 
$E_\% \colon \Gspaces \to Spectra$ 
which is $\G$-homotopy invariant and $\G$-excisive. This will be our method for comparing functors. The \emph{orbit category}
$\Or(G)$ is the category with objects $G/K$,
for $K$ any subgroup of $\G$, and the morphisms
are $\G$-maps.

\section{Change of Coefficients}
We will need some `change of coefficient' properties for the categories defined in the last section. The first three properties are essentially just translations of \cite[Proposition 2.8]{bartels-reich2} into our language.
The corresponding versions for additive categories with involution are
 needed to apply these change of coefficient functors to $L$-theory. 

\begin{definition} Let $K$ and $\G$ be groups, $\cA$ an additive category with commuting right $K$ and $\G$-actions, and $S$ a $K$-$\G$ biset.
Then,  the category $\DAB{S}{K}{\cA}$ has a right $\G$-action via $(g\cdot A)_{y} = g^*A_{yg^{-1}}$ and $(g\cdot\phi)_{y}^{z} = g^*\phi_{yg^{-1}}^{zg^{-1}}$, for all $y, z \in \overline{S}$. We will mostly use the level-preserving subcatetory $\DABl{S}{K}{\cA}$.
\end{definition}
 If $T$ is a left $\G$-set, and $S$ is a transitive $K$-$\G$ biset (meaning that $K\backslash S/G$ is a point), we define  a $K \times \G$-action on $S\times T$ by 
the formula $(k,g)\cdot (s,t):= (ksg^{-1}, gt)$ for all $(k,g)\in K\times \G$ and all $(s,t) \in S \times T$. This action is used in the statements below.

\begin{lemma}\label{zero} Let $T$ be a left $\G$-set, and $S$ be a transitive $K$-$\G$ biset. Then there is an additive functor
$$F\colon \DABl{S \times T\times \Iopen}{K\times \G}{\cA} \to
\DABl{T\times \Iopen}{\G}{\DABl{S}{K}{\cA}}$$
which induces an
equivalence of categories
$$\DABl{S \times T}{K\times \G}{\cA} \simeq 
\DABl{T}{\G}{\DABl{S}{K}{\cA}}\ .$$
\end{lemma}
\begin{proof} We will take the standard resolutions $\wS = K  \times S$, with elements denoted $(k,s)$, for $k\in K$ and $s \in S$, and $\wT = G \times T \times [0,1]$, with elements denoted $(g,t)$, for $g\in G$ and $t \in T\times [0,1]$. Therefore $$\wS \times \wT = K \times G \times S \times T \times [0,1]$$
is a resolution for $S \times T\times [0,1]$.
We define the functor 
$$F\colon \DABl{S \times T\times \Iopen}{K\times \G}{\cA} \to
\DABl{T\times \Iopen}{\G}{\DABl{S}{K}{\cA}}$$ on objects by setting $B =F(A)_{(g,t)}$ in
$\DABl{S}{K}{\cA}$ as the object $B= (B_{(k,s)})$ with $B_{(k,s)} =
A_{(k,g,s,t)}$ in $\cA$. We use a similar formula for morphisms:
$$\left (F(\phi)_{(g,t)}^{(g',t')}\right )_{(k,s)}^{(k',s')} = \phi_{(k,g,s,t)}^{(k',g',s',t')}$$
The proof that this is a well-defined functor is given in Section \ref{appendix},  where step ($5'''$) of the argument depends on
 the assumption that $S$ is a transitive $K$-$\G$ biset.

\medskip
Since $\DABl{S \times T}{K\times \G}{\cA}\simeq 
\DABl{S \times T\times\Iopen}{K\times \G}{\cA}_\emptyset$
and $\DABl{T}{\G}{\DABl{S}{K}{\cA}}\simeq \DABl{T\times\Iopen}{\G}{\DABl{S}{K}{\cA}}_\emptyset$, the functor $F$ induces
 an additive functor
$$F\colon \DABl{S \times T}{K\times \G}{\cA}\to 
\DABl{T}{\G}{\DABl{S}{K}{\cA}}.$$
On this subcategory, we define  an inverse additive functor 
$$F'\colon \DABl{T}{\G}{\DABl{S}{K}{\cA}} \to \DABl{S \times T}{K\times \G}{\cA} 
$$
on objects by setting $ F'(B)_{(k,g,s,t)} = \left ( B_{(g,t)}\right )_{(k,s)}$,
 and a similar formula for morphisms:
 $$F'(\phi)_{(k,g,s,t)}^{(k',g',s',t')} = \left (\phi_{(g,t)}^{(g',t')}\right )_{(k,s)}^{(k',s')}$$
 It is easy to check that $F'$ is a well-defined functor.
 The functors $F$ and $F'$ are inverses, so give an equivalence of categories.
\end{proof}
\begin{corollary}\label{zerocor}
 Let $\G$ and $K$ be groups, and $\cA$ be  an additive category with commuting right $K$ and $\G$-actions,. Then
$$\DAB{\bullet}{K\times \G}{\cA} \simeq \DAB{\bullet}{\G}{\DAB{\bullet}{K}{\cA}}\ .$$
\end{corollary}
\begin{proof} We substitute $S = \bullet$ and $T=\bullet$ in the statement above. Note that morphisms are automatically level-preserving in this case.
\end{proof}
\begin{lemma}\label{one}
Let $K$ and $\G$ be groups, $\cA$ an additive category with commuting right $K$ and $\G$-actions, and $S$ a transitive $K$-$\G$ biset. Then, for any $\G$-CW complex $X$, the functors
$$F^\lambda_{K\times \G}(S\times X;\cA)$$
 and 
 $$F^\lambda_{\G}(X;\DABl{S}{K}{\cA})$$
  are homotopy equivalent, where $\lambda=\KI$ or $\LI$. Here $K\times \G$ acts on $S\times X$ by the formula $(k,g)\cdot (x,s):= ( ksg^{-1}, gx)$.
\end{lemma}

 \begin{proof} By \cite[5.1]{davis-lueck1} it is enough to show that the two functors are $\G$-homotopy invariant, $\G$-excisive, and homotopy equivalent when restricted to the orbit category $\Or(\G)$. For the first two properties, we apply \cite[1.28,~4.2]{cap1}. For the last property, we follow the method of \cite[\S 8]{hp4}.
 Let $T = \G/H$ and consider the following commutative diagram
 $$\xymatrix@C-16pt{\DABl{\hbox to 60pt{$S \times T \times \Iopen$}}{K \times \G}{\cA}_\emptyset \ar[r]\ar[d]^F_\simeq& \DABl{\hbox to 60pt{$S \times T \times \Iopen$}}{K \times \G}{\cA} \ar[r]\ar[d]^F & \DABl{\hbox to 60pt{$S \times T \times \Iopen$}}{K \times \G}{\cA}^{>0}\ar[d]^F\\
 \DABl{\hbox to 42pt{$T\times \Iopen$}}{G}{\DABl{S}{K}{\cA}}_\emptyset\ar[r]
 &
  \DABl{\hbox to 42pt{$T\times \Iopen$}}{G}{\DABl{S}{K}{\cA}}\ar[r]&
   \DABl{\hbox to 42pt{$T\times \Iopen$}}{G}{\DABl{S}{K}{\cA}}^{>0}
   }
 $$ 
 where the vertical maps are induced by the additive functors of Lemma \ref{zero}. We apply  $\lambda=\KI$ or $\lambda=\LI$ to obtain fibrations of spectra. 
 Note that $\lambda$ applied to either of the middle two categories gives a spectrum with trivial homotopy groups (by an Eilenberg swindle). Therefore the first and third vertical maps induce a homotopy equivalence of spectra.
 Since the level-preserving condition is automatic on the germ categories, we are done.
 \end{proof}

The next  property allows us to divide out a normal subgroup in suitable circumstances. 
\begin{lemma}\label{two}
Let $N$ be a normal subgroup of $\G$, and $\cA$ be an additive category with right $\G$-action such that $N$ acts trivially. Let $X$ be a $\G$-CW complex such that $N$ acts freely on $X$. Then there is an additive functor
$$\DAcatH{X}{\G}{\cA}\  \to\  \DAcatH{N\backslash X}{\G/N}{\cA}$$
which induces an isomorphism on $K$-theory after taking germs away from the empty set.
\end{lemma}
\begin{proof}
We will construct a functor $F = F_2\circ F_1$ inducing this isomorphism in two steps. First, 
we have a functor
$F_1\colon \DAcatH{X}{\G}{\cA}  \to \DAcatH{N\backslash X}{\G}{\cA}$, which is the identity on objects and morphisms. The continuous control condition measured in $X$ is stronger than the continuous control condition measured in $N\backslash X$, so this is well-defined. This functor induces a homotopy invariant and $G$-excisive functor
$$F_1\colon \DAcatH{G/H}{\G}{\cA}^{>0} \to \DAcatH{N\backslash G/H}{\G}{\cA}^{>0}$$
for $X = G/H$, and
an equivalence $\DAB{G/H}{G}{\cA} \simeq \DAB{N\backslash G/H}{G}{\cA}$. Therefore $F_1$ induces isomorphisms on $K$-theory after
taking germs away from the empty set (as in the proof of Lemma \ref{one}). Secondly, there is  a functor
 $$F_2\colon \DAcatH{N\backslash X}{\G}{\cA}  \to \DAcatH{N\backslash X}{\G/N}{\cA}$$ 
defined on objects by $F_2(A)_{(gN, \bar y)} = A_{(g, \bar y)}$, where $\bar y \in N\backslash X\times \Iopen$. We define the functor on morphisms by $F_2(\phi)_{(gN, \bar y)}^{(g'N, \bar y')} = \phi_{(g, \bar y)}^{(g', \bar y')}$. This is well-defined by $G$-invariance of the objects and morphisms in the domain, and the continuous control conditions on morphisms agree since both are measured in $N\backslash X$. We also have an inverse functor
 $F_2'$ defined by $F_2'(A)_{(e, \bar y)} = A_{(eN, \bar y)}$ on objects, extended by $G$-equivariance, and similarly for morphisms.
 It follows that $F_2$ is an equivalence of categories.
\end{proof}

In the next statement, if $\cA$ is an additive $\G$-catgeory, we denote by $\Res_H\cA$ the same category considered as an $H$-category under restriction to a subgroup $H$ of $\G$. The following  is ``Shapiro's Lemma" in our setting.

\begin{proposition}\label{three}
Let $H$ be a subgroup of $\G$, $\cA$ be an additive category with $\G$-action, and $X$ be an $H$-CW complex. There is an additive functor
$$\DAB{X\times [0,1)}{H}{\Res_H \cA} \to 
\DAB{G \times_H X\times [0,1)}{\G}{\cA}$$
which induces an equivalence of categories after taking germs.
\end{proposition}

\begin{proof}
 This proposition is proven in~\cite[Proposition 8.3]{bartels-farrell-jones-reich1} in the case where $\cA$ is the category of finitely generated free $R$-modules. The same proof works for any coefficient category once the functor $\Ind\colon  \DAB{X\times [0,1)}{H}{\Res_H \cA} \to 
\DAB{G \times_H X\times [0,1)}{\G}{\cA}$ is defined for general $\cA$. Let $\phi\colon A\to B$ be a morphism in $\DAB{X\times [0,1)}{H}{\Res_H \cA}$. Then
\[\Ind\colon \DAB{X\times [0,1)}{H}{\Res_H \cA}\to \DAB{G \times_H X\times [0,1)}{\G}{\cA}\] is defined by $\Ind(A)_{[g,y]}=(g^{-1})^*A_{y}$, and $\Ind(\phi)_{[g,y]}^{[g',y']}=(g^{-1})^*\phi_{y}^{g^{-1}g'y'}$ if $g^{-1}g'\in H$, and is zero otherwise. The inverse of this functor on the corresponding germ categories is induced by the inclusion $i\colon X\to G \times_H X$. That is, $\Ind^{-1}(M)_y=M_{i(y)}$ and $
\Ind^{-1}(\psi)_y^{y'}=\psi_{i(y)}^{i(y')}$.
\end{proof}

\begin{remark}
The equivalences given in these three properties are natural with respect to equivariant maps $X \to X'$. If $\cA$ is an additive category with involution, one can check that the above properties continue to hold in this context. 
This is needed for applications to the $L$-theory assembly maps.
\end{remark}

\section{Assembly and subgroups}
The properties of the continuously controlled categories given so far lead to a formal statement about assembly and subgroups. This is just our version of \cite[Proposition 4.2]{bartels-reich2}. If $H$ is a subgroup of $\G$, and $\cA$ is an additive $H$-category, we denote
$\Ind_H^\G\cA := \DABl{\G}{H}{\cA}$ considered as a $\G$-category by using the $H$-$\G$ biset structure of $\G$.

\begin{proposition}\label{prop:subgroups}
Let $f\colon X \to X'$ be a $\G$-equivariant map between $\G$-CW complexes. Let $H$ be a subgroup of $\G$, and let $\cA$ be an additive category with $H$-action. Then there is a commutative diagram
$$\xymatrix
{\DAB{\Res_H X\times [0,1)}{H}{\cA}^{>\emptyset}
\ar[r]^{f_*}\ar[d]^{\simeq}&\DAB{\Res_H X'\times [0,1)}{H}{\cA}^{>\emptyset}\ar[d]^{\simeq}\\
\DAB{X\times [0,1)}{\G}{\Ind_H^\G\cA}^{>\emptyset}\ar[r]^{f_*}\ar[u]&{\DAB{X'\times [0,1)}{\G}{\Ind_H^\G\cA}}^{>\emptyset}\ar[u] }
$$
\end{proposition}

\begin{proof}
By Lemma~\ref{one} with $K=H$ and $S=\G$, we have 
$$\DAB{X\times [0,1)}{\G}{\Ind_H^\G\cA}^{>\emptyset}\simeq \DAB{\G\times X\times [0,1)}{H\times \G}{\cA}^{>\emptyset}$$
where $1\times G$ acts trivially on $\cA$ in the right-hand side.
Finally,
$$\DAB{\G\times X\times [0,1)}{H\times \G}{\cA}^{>\emptyset} \simeq \DAB{\Res_H X\times [0,1)}{H}{\cA}^{>\emptyset}$$
by applying Lemma~\ref{two} to $H\times \G$ with $N = \G$. Note that $G$ acts freely on $G \times X$, with quotient isomorphic to $\Res_H X$.
\end{proof}

\begin{corollary} Let $H$ be a subgroup of $\G$ and $\cF$ be a family of subgroups of $G$. Suppose that the $K$-theory or $L$-theory $(\G,\cF,\cB)$-assembly map is an isomorphism (respectively injection or surjection) for every  additive coefficient category $\cB$ with $\G$-action. Then the $(H,\cF|_H,\cA)$-assembly map is an isomorphism (respectively injection or surjection) for any  additive coefficient category $\cA$ with $H$-action.
\end{corollary}
\begin{proof}
Just substitute $X=E_\cF G$ and $X'=\bullet$ in the diagram above.
\end{proof}

In particular, this says that the Farrell-Jones conjecture with coefficients is stable under taking subgroups. These ideas can be extended further to obtain a version of the fibered isomorphism conjecture.

\begin{proposition}\label{maintwo}
Let $\phi\colon H\to G$ be a group homomorphism, and let $\cF$ be a family of subgroups of $G$. If the $K$-theory or $L$-theory assembly map for $G$ relative to the family $\cF$ is an isomorphism
(respectively injective or surjective), with twisted coefficients in any additive $G$-category, then the assembly map
for $H$ relative to the pull-back family $\phi^*\cF=\{K\leq H\;|\;\phi(K)\in \cF\}$ is an isomorphism (respectively injection or surjection), with twisted coefficients in any additive $H$-category.
\end{proposition}
\begin{proof}
The proof is the same as for Proposition~\ref{prop:subgroups} using $X=E_\cF G$ and $X'=\bullet$, with the action of $H$ on $S=G$ and on $X$ defined via $\phi$, and $\Res_\phi X = E_{\phi^*\cF} G$.
\end{proof}
\section{Assembly for Extensions}
In \cite{oyono-oyono1} the Baum-Connes conjecture for topological $K$-theory  is shown to pass to extensions. We show that there is a similar statement for algebraic $K$- and $L$-theory.

The proof outline used in  \cite{oyono-oyono1} has two main steps, which we now translate into our setting. In the first step we use a discrete transitive right $G$-set $S$, which can be expressed as a single orbit $S = \{s\}\cdot G$.

\begin{proposition}\label{step1}
Let $X$ be a $\G$-CW complex, $S=\{s\}\cdot G$, and $\cA$ be an additive $\G$-category with involution. Then there is an additive functor
\[ \DAB{\Res_{G_s} X\times [0,1)}{G_s}{\Res_{G_s} \cA} \to \DAB{X \times \Iopen}{\G}{\DABl{S}{0}{\cA}}^{>0} \]
which induces a homotopy equivalence of spectra after applying $\KI$ or $\LI$. This equivalence is natural with respect to maps $X\to X'$ of $\G$-CW complexes.
\end{proposition}

\begin{proof}
By Proposition~\ref{three},
$$\KI(\DAB{\Res_{G_s}X\times [0,1)}{G_s}{\Res_{G_s} \cA}^{>0}) \simeq
\KI(\DAB{G \times_{G_s} X\times [0,1)}{\G}{\cA}^{>0}).$$
Since $G \times_{G_s} X$ is $G$-equivariantly homeomorphic to $(G_s\backslash G) \times X= S\times X$,  via the map $[g,x] \mapsto (Hg^{-1}, gx)$, and so
\[ \DAB{G \times_{G_s} X\times [0,1)}{\G}{\cA}^{>0} \cong \DAB{S\times X \times \Iopen}{\G}{\cA}^{>0},  \]
 where $S\times X$ has the usual left $G$-action $g\cdot(s,x) = (sg^{-1}, gx)$.
Finally, by Lemma \ref{one},
$$\KI(\DAB{S\times X \times \Iopen}{\G}{\cA}^{>0}) \simeq\KI(\cD^G(X \times \Iopen;\DABl{S}{0}{\cA})^{>0}).$$
The same proof works if we replace $\KI$ by $\LI$.
\end{proof}

\begin{example}\label{ab}
Let $\pi\colon \G \to K$ be a surjection of groups, and $V\subset K$ be a subgroup. We consider $S=K$ as a right-$(\G\times V)$-set via the transitive action
$k\cdot (g,v) := \pi(g)^{-1}kv$, where $g\in \G$, $v\in V$, and $k\in K$. Let $X$ be a $(\G\times K)$-CW complex, and let $V'\subset \G\times V$ denote the stabilizer subgroup of $e\in K$. Notice that $V'\cong \pi^{-1}(V)$, since $\pi(g)^{-1}v = e$ implies $g \in \pi^{-1}(v)$.
By Proposition~\ref{step1}, we have a commutative diagram
$$\xymatrix
{F^\lambda_{V'}(X;\cA)  
\ar[r]\ar[d]^{\simeq}&F^\lambda_{V'}(\bullet;\cA)
\ar[d]^{\simeq}\\
F^\lambda_{\G\times V}(X;\DABl{K}{0}{\cA})\ar[r]&F^\lambda_{\G\times V}(\bullet;\DABl{K}{0}{\cA})}
$$
for $\lambda=\KI$ or $\lambda=\LI$, which shows that the lower assembly map is a homotopy equivalence of spectra whenever the
upper map is an equivalence.
\end{example}

\begin{remark}\label{cd}
In the proof of Theorem~A, we will be using Example~\ref{ab} with $X=E_{\cF_\G}\G\times E_{\cF_K}K$, where $\cF_\G$ is a family of subgroups of $\G$ and $\cF_K$ is a family of subgroups of $K$ such that $\pi(H)\in \cF_K$ for every $H\in \cF_\G$. If $V\in \cF_K$, then the map $E_{\cF_\G}\G\times E_{\cF_K}K\to E_{\cF_\G}\G\times \bullet$ is a $\G\times V$-equivariant homotopy equivalence. Therefore, it is a $V'$-equivariant homotopy equivalence. Since $V'\cong \pi^{-1}(V)$, we have the homotopy commutative diagram:
 $$\xymatrix
{F^\lambda_{\pi^{-1}(V)}(E_{\cF_\G}\G;\cA)  
\ar[r]^{a}\ar[d]^{\simeq}&F^\lambda_{\pi^{-1}(V)}(\bullet;\cA)
\ar[d]^{\simeq}\\
F^\lambda_{\G\times V}(X;\DABl{K}{0}{\cA})\ar[r]^{b}&F^\lambda_{\G\times V}(\bullet;\DABl{K}{0}{\cA})}
$$
where $X=E_{\cF_\G}\G\times E_{\cF_K}K$.

If $V=K$, then $\G\cong V'\subset \G\times K$ and $\G$ acts on $X=E_{\cF_\G}\G\times E_{\cF_K}K$ via this isomorphism. Since we are assuming that $\pi(H)\in \cF_K$ for every $H\in \cF_\G$, $X$ is a model for $E_{\cF_\G}\G$. Thus, we have the homotopy commutative diagram:
 $$\xymatrix
{F^\lambda_{\G}(E_{\cF_\G}\G;\cA)  
\ar[r]^{c}\ar[d]^{\simeq}&F^\lambda_{\G}(\bullet;\cA)
\ar[d]^{\simeq}\\
F^\lambda_{\G\times K}(X;\DABl{K}{0}{\cA})\ar[r]^{d}&F^\lambda_{\G\times K}(\bullet;\DABl{K}{0}{\cA})}
$$
\end{remark}

\begin{definition} Let $\G_1$ and $\G_2$ be discrete groups, and let
$X_1$ and $X_2$ be $\G_1$- and $\G_2$-CW complexes, respectively. Let $\cA$ be a $\G_1\times \G_2$-additive category with involution. The \emph{partial assembly map},
$$\mu^{\G_1,\G_2}\colon 
F^\lambda_{\G_1\times \G_2}(X_1\times X_2;\cA) \to F^\lambda_{\G_2}(X_2;\DAB{\bullet}{\G_1}{\cA}),
$$
 is the map induced by the second factor projection $X_1\times X_2 \to \bullet \times X_2$, composed with the homotopy equivalence from Lemma~\ref{one} with $S=\bullet$.
\end{definition}

\begin{lemma}
The partial assembly map is natural in the control spaces and involution invariant. \qed
\end{lemma}
Now the second step of the proof outline gives a criterion for the partial assembly map to be an equivalence.
\begin{proposition}\label{step2}
Let $\G$ and $K$ be groups, and let $\cB$ be an additive
$\G\times K$-category. Let $\cF_K$ be a family of subgroups of $K$. Let $X_1$ be a $\G$-CW complex and $X_2$ be a $K$-CW complex with isotropy in $\cF_K$. Suppose that
$$ F^\lambda_{\G\times V}(X_1\times \bullet;\cB) \to F^\lambda_{\G\times V}(\bullet;\cB)$$
is a homotopy equivalence for all subgroups $V\in \cF_K$. Then the partial assembly map
$$\mu^{\G,K}\colon F^\lambda_{\G\times K}(X_1\times X_2;\cB) \to F^\lambda_{K}(X_2;\DAB{\bullet}{\G}{\cB})$$
is also an equivalence for $\lambda=\KI$ or $\lambda=\LI$.
\end{proposition}
\begin{proof}
Suppose that $X_2=K/V$ for some $V\in \cF_K$. Then, by Shapiro's Lemma  (Proposition~\ref{three}),
$$\xymatrix
{F^\lambda_{\G\times V}(X_1\times \bullet;\cB)
\ar[r]^{\mu^{\G,V}}\ar[d]^{\simeq}&\ F^\lambda_V(\bullet;\DAB{\bullet}{\G}{\cB})
\ar[d]^{\simeq}\\
F^\lambda_{\G\times K}(X_1\times K/V;\cB)\ar[r]^{\mu^{\G,K}}&\ F^\lambda_{\G\times K}(K/V;\DAB{\bullet}{\G}{\cB}) }
$$
and the upper map is an equivalence by assumption, since
$F^\lambda_V(\bullet;\DAB{\bullet}{\G}{\cB})\simeq
F^\lambda_{\G\times V}(\bullet;\cB)$.
The functors $ H(X_2) :=F^\lambda_{\G\times K}(X_1\times X_2;\cB)$ and $H'(X_2):= F^\lambda_{K}(X_2;\DAB{\bullet}{\G}{\cB})$ are homotopy-invariant and $K$-excisive functors from $K$-CW complexes to spectra. Since $H(K/V) \simeq H'(K/V)$ for all $V\in \cF_K$, we conclude that $H(X_2) \simeq H'(X_2)$ for all $K$-CW complexes with isotropy in $\cF_K$.
\end{proof}

The following is our main result about extensions:
\begin{theorem}\label{mainone}
Let $N \to G \xrightarrow{\pi} K$ be a group extension, where $N\triangleleft\, G$ is a normal subgroup, and $K$ is the quotient group. Let $\cF_G$ be a family of subgroups of $G$ and $\cA$ an additive category with right $G$-action. Let $\cF_K$ be a family of subgroups of $K$ such that $\pi(H)\in \cF_K$ for every $H\in \cF_G$. Suppose that for every $V\in \cF_K$ the $(\pi^{-1}(V),\cF_G|_{\pi^{-1}(V)},\cA)$-assembly map in algebraic K-theory is an isomorphism, and that for every additive category $\cB$ with right $K$-action the $(K,\cF_K,\cB)$-assembly map in algebraic $K$-theory is injective (resp. surjective). Then the $(G,\cF_G,\cA)$-assembly map in algebraic $K$-theory is injective (resp. surjective).

The same statement holds for algebraic $L$-theory as well.
\end{theorem}
\begin{example}\label{exone}
Suppose that  $N$ is  \emph{finite}  normal subgroup of $G$. Then the Farrell-Jones conjecture with twisted coefficients holds for $G$  if it holds for $K=G/N$.
\end{example}
\begin{example}\label{extwo}
 Suppose that $1\to N \to G \to K \to 1$ is a group extension, and $\cF_G$ and $\cF_K$ both denote the family of finite subgroups of their respective groups. Then the conclusions of Theorem \ref{mainone} hold provided that the assembly map is injective (resp. surjective) for $K$ and for every subgroup of $G$ containing $N$ as a subgroup of finite index.
\end{example}
\begin{proof}[The Proof of Theorem \ref{mainone}] 
Let $X=E_{\cF_\G}\G\times E_{\cF_K}K$. Let $V\in \cF_K$ be given. By Remark~\ref{cd}, we have a homotopy commutative diagram:
$$\xymatrix
{F^\lambda_{\pi^{-1}(V)}(E_{\cF_\G}\G;\cA)  
\ar[r]^{a}\ar[d]^{\simeq}&F^\lambda_{\pi^{-1}(V)}(\bullet;\cA)
\ar[d]^{\simeq}\\
F^\lambda_{\G\times V}(X;\DABl{K}{0}{\cA})\ar[r]^{b}&F^\lambda_{\G\times V}(\bullet;\DABl{K}{0}{\cA})}
$$
Let $\cB = \DABl{K}{0}{\cA}$, and note that the upper map $a$ is an equivalence by assumption,  since $\Res_{\pi^{-1}(V)}E_{\cF_\G}\G$ is a universal space for the family $\cF_G|_{\pi^{-1}(V)}$. Hence, the lower map $b$ is also an equivalence.
By Proposition~\ref{step2}, we have the homotopy commutative diagram:
$$\xymatrix{F^\lambda_{\G\times K}(X;\cB)  
\ar[r]^{d}\ar[d]_{\mu^{\G,K}}^{\simeq}&
F^\lambda_{\G\times K}(\bullet;\cB)\ar[d]^{\simeq}\\
F^\lambda_{K}(E_{\cF_K}K;\DAB{\bullet}{\G}{\cB})\ar[r]^e&
F^\lambda_K(\bullet;\DAB{\bullet}{\G}{\cB})}$$
By assumption, the map $e$ is injective (resp. surjective), which implies that $d$ is injective (resp. surjective).

Using Remark~\ref{cd} again, we have the homotopy commutative diagram: 
$$\xymatrix
{F^\lambda_{\G}(E_{\cF_\G}\G;\cA)  
\ar[r]^{c}\ar[d]^{\simeq}&F^\lambda_{\G}(\bullet;\cA)
\ar[d]^{\simeq}\\
F^\lambda_{\G\times K}(X;\DABl{K}{0}{\cA})\ar[r]^{d}&F^\lambda_{\G\times K}(\bullet;\DABl{K}{0}{\cA})}
$$
Therefore, the assembly map $c$ is injective (resp. surjective).
\end{proof}

\begin{corollary}\label{vproducts}
The Farrell-Jones conjecture with twisted coefficients is true for $\G_1\times \G_2$ if and only if it is true for $\G_1$,  $\G_2$, and every product
$V_1 \times V_2$, where $V_1\leq \G_1$ and $V_2\leq \G_2$ are virtually cyclic subgroups. 
\end{corollary}

\begin{proof} By our main result applied to the projection $G_1\times G_2  \to G_2$, we may assume that $G_2$ is virtually cyclic. Similarly, we may assume that $G_1$ is virtually cyclic.  Thus, we are reduced to knowing the conjecture for products $V_1 \times V_2$ of virtually cyclic subgroups of $G_1$ and $G_2$ respectively.
\end{proof}
\begin{remark}
A product $V_1 \times V_2$ of virtually cyclic subgroups can be further reduced to the basic cases $\bZ \times \bZ$, $\bZ \times D_\infty$ and $D_\infty \times D_\infty$ after quotients by finite normal subgroups. 
\end{remark}

\section{The proof of Lemma \ref{zero}}\label{appendix}

 We will check the details of Lemma \ref{zero}, which asserts that there is an additive functor $$F\colon \Dlcat{S\times T}{K\times G}{\cA}\to 
\Dlcat{T}{G}{\DABl{S}{K}{\cA}}$$
defined by
\begin{align}
	& (F(A)_{(g,t)})_{(k,s)} := A_{(k,g,s,t)} \notag\\
	& \Big(F(\phi)_{(g,t)}^{(g',t')}\Big)_{(k,s)}^{(k',s')} :=\phi_{(k,g,s,t)}^{(k',g',s',t')}. \notag
\end{align}
Here $\cA$ is an additive category with commuting right $K$ and $G$-actions, $T$ a left $G$-set and $S$ a transitive $K$-$G$ biset.
The group $K\times G$ acts on $S\times T$ by the formula $(k,g)\cdot (s,t):= (ksg^{-1}, gt)$. Recall the notation $(k,s)$ for elements of $K\times S$, and 
$(g, t)$ for elements of $G \times T \times [0,1]$. We will let 
$\epsilon\colon T\times [0,1] \to T$ denote the projection map. Notice that $\phi_{(k,g,s,t)}^{(k',g',s',t')} = 0$ unless $s=s'$ and $\epsilon(t)=\epsilon(t')$, since the morphisms $\phi\colon A\to B$ in the domain category are assumed to be level-preserving.
The free $(K\times G)$-space
 $$\wS \times \wT = K \times G \times S \times T \times [0,1]$$
is a resolution for $S \times T\times [0,1]$.
The proof that $F$ is a functor is done in the following steps.

\nr{1} {$F(\phi\circ \psi)=F(\phi)\circ F(\psi)$}. 
Since $$\Big(F(\phi)\circ F(\psi)\Big)_{(g,t)}^{(g',t')}=\sum_{(g'',t'')}F(\phi)_{(g'',t'')}^{(g',t')}\circ F(\psi)_{(g,t)}^{(g'',t'')}$$
we have that:
\begin{align*}
	\Big(\big(F(\phi)\circ F(\psi)\big)_{(g,t)}^{(g',t')}\Big)_{(k,s)}^{(k',s')} & \;=\; \Big(\sum_{(g'',t'')}F(\phi)_{(g'',t'')}^{(g',t')}\circ F(\psi)_{(g,t)}^{(g'',t'')}\Big)_{(k,s)}^{(k',s')} \\
	& \;=\; \sum_{(g'',t'')}\Big(F(\phi)_{(g'',t'')}^{(g',t')}\circ F(\psi)_{(g,t)}^{(g'',t'')}\Big)_{(k,s)}^{(k',s')} \\
	& \;=\; \sum_{(g'',t'')}\sum_{(k'',s'')}\phi_{(k'',g'',s'',t'')}^{(k',g',s',t')}\circ \psi_{(k,g,s,t)}^{(k'',g'',s'',t'')} \\
	& \;=\; (\phi\circ \psi) _{(k,g,s,t)}^{(k',g',s',t')} \\
	& \;=\; \Big(F(\phi\circ \psi)_{(g,t)}^{(g',t')}\Big)_{(k,s)}^{(k',s')}
\end{align*}

\nr{2} \emph{$F(A)_{(g,t)}$ is an object of $\DABl{S}{K}{\cA}$, 
 for every $(g,t)\in G\times T\times [0,1)$.}

\nr{2'} \emph{$F(A)_{(g,t)}$ is $K$-invariant.} For each $h \in K$, 
\begin{align*}
	\big(h^*(F(A)_{(g,t)})\big)_{(k,s)} & \;=\; h^*\big((F(A)_{(g,t)})_{(hk,hs)}\big) \\
	& \;=\; h^*\big(A_{(hk,g,hs,t)}\big) \\
	& \;=\; (h^*A)_{(k,g,s,t)} \\
	& \;=\; A_{(k,g,s,t)} \\
	& \;=\; (F(A)_{(g,t)})_{(k,s)}
\end{align*}

\nr{2''} \emph{The support of $F(A)_{(g,t)}$ is $K$-compact in $K\times S$.}

Since a discrete $K$-set is $K$-compact if and only if its image under the quotient map is finite, we need to show that $K\backslash \supp(F(A)_{(g,t)})$ is finite. Let $p$ be the projection map from $K\times G\times S\times T\times[0,1)$ to $K\times G\times S\times T$, $M=p\big(\supp(A)\big)$, and $N=p\big(\supp(A)\cap K\times \{g\}\times S\times \{t\}\big)\subset M$. 
Consider the following commutative diagram, in which $f(k',g',s',t')=(k',s'g')$, $m_{g}(k,s)=(k,sg^{-1})$, and the vertical arrows are quotient maps.
\[ \xymatrix
{ K\times G\times S\times T \ar[r]^-{f}\ar[d]^{q_{K\times G}} & K\times S \ar[r]^-{m_{g}}\ar[d]^{q_{K}} & K\times S\ar[d]^{q_{K}}\\
(K\times G)\backslash(K\times G\times S\times T) \ar[r]^-{\bar{f}} & K \backslash(K\times S) \ar[r]^-{\bar{m}_{g}} & K \backslash(K\times S)  }
\]
Since $M$ is discrete and $(K\times G)$-compact, $q_{K\times G}(M)$ is finite. Since $N\subset M$, $q_{K\times G}(N)$ is also finite. Therefore, $(\bar{m}_{g}\circ \bar{f}\circ q_{K\times G})(N)=(q_K\circ m_{g}\circ f)(N)=q_K(\supp(F(A)_{(g,t)}))$ is finite.

\nr{3} \emph{$F(\phi)_{(g,t)}^{(g',t')}$ is a morphism of $\DABl{S}{K}{\cA}$, for every $(g,t),(g',t')\in G\times T\times [0,1)$.}

\nr{3'} \emph{$F(\phi)_{(g,t)}^{(g',t')}$ is $K$-invariant.}
	The proof is similar to the proof of \nr{2'}

\vspace{.25cm}
	\nr{3''} \emph{Fix $(k,s)\in K\times S$. Then, the following set is finite:
\[ P=\Big\{ (k',s')\in K\times S \, \Big| \, \Big(F(\phi)_{(g,t)}^{(g',t')}\Big)_{(k,s)}^{(k',s')}\neq 0 \text{\ or\ } \Big(F(\phi)_{(g,t)}^{(g',t')}\Big)_{(k',s')}^{(k,s)}\neq 0 \Big\}. \]}
The sets $\Big\{ (k',g',s',t')\in K\times G\times S\times T\times[0,1)\, \Big| \, \phi_{(k,g,s,t)}^{(k',g',s',t')}\neq 0 \Big\}$
 and $\Big\{ (k',g,s',t)\in K\times G\times S\times T\times[0,1)\, \Big| \, \phi_{(k',g,s',t)}^{(k,g',s,t')}\neq 0 \Big\}$
are finite and their union projects onto $P$.

\vspace{.25cm}
	\nr{3'''} \emph{$F(\phi)_{(g,t)}^{(g',t')}$ is level preserving in $S$.} This is because $\phi$ is level-preserving in $S\times T$.

\nr{4} \emph{$F(A)$ is an object of $\Dlcat{T}{G}{\DABl{S}{K}{\cA}}$.}

\nr{4'} \emph{$F(A)$ is $G$-invariant.} For each $\gamma \in G$,
\begin{align*}
	\big(\gamma^*(F(A))_{(g,t)}\big)_{(k,s)} & \;=\; \big(\gamma^*(F(A)_{(\gamma g,\gamma t)})\big)_{(k,s)} \\
	& \;=\; \gamma^*\big((F(A)_{(\gamma g,\gamma t)})_{(k,s\gamma^{-1})}\big) \\
	& \;=\; \gamma^*\big(A_{(k,\gamma g,s\gamma^{-1},\gamma t)}\big)  \\
	& \;=\; (\gamma^*A)_{(k,g,s,t)}\\
	& \;=\; A_{(k,g,s,t)} \\
	& \;=\; (F(A)_{(g,t)})_{(k,s)}
\end{align*}
	
\nr{4''} \emph{The support of $F(A)$ is relatively $G$-compact in $G\times T\times [0,1)$.}
	
Let $p\colon  K\times G\times S\times T\times [0,1)\to G\times T\times [0,1)$ be the projection map. Since $\supp(A)$ is relatively $(K\times G)$-compact and $p(\supp(A))=\supp(F(A))$, $\supp(F(A))$ is relatively $G$-compact in $G\times T\times [0,1)$.

\nr{4'''}\emph{The support of $F(A)$ is locally finite in $G\times T\times [0,1)$.} 

Let $(g,t)\in \supp(F(A))$ be given. We must find an open neighborhood $U\subset G\times T\times [0,1)$ of $(g,t)$ such that $U\cap \supp(F(A))=\{(g,t)\}$. Let
$$L= \{ (k,s)\in K\times S \; | \; (k,g,s,t)\in \supp(A) \}.$$
From {\bf (${\mathbf 1''}$)}, we know that $L$ is $K$-compact. That is, $L=K\cdot(K_0\times S_0)$, where $K_0\subset K$ and $S_0\subset S$ are finite sets. Since $\supp(A)$ is locally finite in $K\times G\times S\times T\times [0,1)$, for each $(k_i,s_i)\in K_0\times S_0$, there is a neighborhood $U_i\subset T\times [0,1)$ of $t$, such that 
$$(\{k_i\}\times \{g\}\times \{s_i\} \times U_i)\cap \supp(A)=\{(k_i,g,s_i,t)\}.$$
Thus, for each $(k,s)\in L$, there is an $i$, such that 
$$(\{k\}\times  \{g\} \times \{s\} \times  U_i)\cap \supp(A)=\{(k,g,s,t)\}.$$
Therefore, if we let $U=  \{g\} \times  (\cap_iU_i)$, then $U\cap \supp(F(A))=\{(g,t)\}$.

\nr{5} \emph{$F(\phi)$ is a morphism in $\Dlcat{T}{G}{\DABl{S}{K}{\cA}}$.}

\nr{5'} \emph{$F(\phi)$ is $G$-invariant.} The proof is similar to the proof of \nr{3'}

\nr{5''} \emph{Fix $(g,t)\in G\times T\times [0,1)$. Then, the following set is finite  
\[ \Big\{ (g',t')\in G\times T\times [0,1) \, \Big| \, F(\phi)_{(g,t)}^{(g',t')}\neq 0 \text{\ or\ } F(\phi)_{(g',t')}^{(g,t)}\neq 0 \Big\}.\]}

As we saw in {\bf (${\mathbf 2''}$)}, $\supp(A)\cap K\times \{g\} \times S\times \{t\}$ is $K$-compact. Therefore, it is contained in $K\cdot (K_0\times \{g\} \times S_0\times  \{t\})$, for some finite subsets $K_0\in K$ and $S_0\in S$. Notice that by $K$-equivariance, $F(\phi)_{(g,t)}^{(g',t')}\neq 0$ if and only if there exists an $s_0\in S_0$, $k_0\in K_0$ and $k'\in K$ such that $\phi_{(k_0,g,s_0,t)}^{(k',g',s,t')}\neq 0$. But for each $k_0\in K_0$ and each $s_0\in S_0$, there are only finitely many $k'\in K$, $g'\in G$ and $t'\in T\times [0,1)$ such that $\phi_{(k_0,g,s_0,t)}^{(k',g',s_0,t')}\neq 0$. Therefore, there are only finitely many $g'\in G$ and $t'\in T\times [0,1)$ such that $F(\phi)_{(g,t)}^{(g',t')}\neq 0$. A similar argument shows that there are only finitely many $g'\in G$ and $t'\in T\times  [0,1)$ such that $F(\phi)_{(g',t')}^{(g,t)}\neq 0$.

\nr{5'''} \emph{$F(\phi)$ is continuously controlled in $T\times [0,1)$.}
	
Let $\phi\colon A\to B$ be a morphism in $\Dlcat{S\times T}{K\times G}{\cA}$. Let $(x_0,1)\in T\times [0,1]$ and a $G_{x_0}$-invariant neighborhood $U\subset T\times [0,1]$ of $(x_0,1)$ be given. We must find a $G_{x_0}$-invariant neighborhood $V\subset T\times [0,1]$ of $(x_0,1)$, such that $F(\phi)_{(g,t)}^{(g',t')}=0=F(\phi)^{(g,t)}_{(g',t')}$ whenever $(g,t)\in G\times V$ and $(g',t')\notin G\times U$.

By definition, $\Big(F(\phi)_{(g,t)}^{(g',t')}\Big)_{(k,s)}^{(k',s)}=\phi_{(k,g,s,t)}^{(k',g',s,t')}$. Let $s_0\in S$ with $K\cdot s_0\cdot G = S$, and let $H\leq K\times G$ be the stabilizer subgroup of $s_0$. We will identify $G \times T\times [0,1]$ with the level $\{s_0\} \times G \times T\times [0,1]$.  Notice that the intersection of $\supp(A)$ with $K\times G\times \{s_0\} \times T\times [0,1)$ is contained in, 
\[ \bigcup_{(a,b)\in H} a\cdot K_0 \times b\cdot G_0\times \{s_0\} \times b\cdot T_0\times [0,1),  \] 
where $K_0\subset K$, $G_0\subset G$ and $T_0\subset T$ are finite sets. This holds since $\supp(A)$ is relatively $(K\times G)$-compact and any element of $(K\times G)-H$ will move $s_0$ to another level in $S$.

Suppose that $\phi_{(k,g,s,t)}^{(k',g',s,t')}\neq 0$ for some $k\in K$, $g\in G$ and $t\in U$. Then we can write $\tau s\gamma^{-1}=s_0$, for some $\tau  \in K$ and some $\gamma \in G$. 
By equivariance, $\phi_{(\tau k,\gamma g,s_0,\gamma t)}^{(\tau k',\gamma g',s_0,\gamma t')}\neq0$.  For this to happen, 
$(\tau k,\gamma g,s_0,\gamma t)\in \supp(A)$. This implies that there exists $(a,b) \in H$ such that $$(\tau k,\gamma g,s_0,\gamma t)\in a\cdot K_0 \times b\cdot G_0\times \{s_0\} \times b\cdot T_0\times [0,1),$$ which is equivalent to saying that
\[ (a^{-1}\tau k,b^{-1}\gamma g,s_0,b^{-1}\gamma t)\in K_0 \times G_0\times \{s_0\} \times T_0\times [0,1) \]
In particular, $b^{-1}\gamma t \in b^{-1}\gamma U \cap  (T_0\times [0,1))$.

Since $T_0$ is finite, there are only finitely many elements of $G$, say $\{g_1,g_2,\dots,g_r\}$, such that $g_i U \cap (T_0\times [0,1)) \neq \emptyset$. Therefore, 
$\gamma=bg_i$ for some $(a,b)\in H$ that fixes $s_0$ and some $i$ with $1\leq i \leq r$. 

Since $\phi$ is continuously controlled at $g_i\cdot (x_0,1)$ along $S\times T\times 1$, there is a neighborhood $V_i\subset T\times [0,1]$ of $( x_0,1)$ such that $\phi_{(k,g,s_0,g_it)}^{(k',g',s_0,g_it')}=0$ if $t\in  V_i$ and $t'\notin U$, for $1\leq i \leq r$.

Let $V=\cap_i V_i$. Then, if $t\in  V$ and $t'\notin U$, we have 
$$ \phi_{(a^{-1}\tau k,g_ig,s_0,g_it)}^{(a^{-1}\tau k',g_ig',s_0,g_it')}=0$$
and hence 
$$0 = \phi_{(\tau k,bg_ig,s_0,bg_it)}^{(\tau k',bg_ig',s_0,bg_it')} = 
\phi_{(\tau k,\gamma g,s_0,\gamma t)}^{(\tau k',\gamma g',s_0,\gamma t')} =
\phi_{(k,g,s,t)}^{(k',g',s,t')} $$
by equivariance of the morphisms, and the relations $\gamma=bg_i$, $s_0 = \tau s\gamma^{-1}$.
A similar argument shows that $F(\phi)^{(g,t)}_{(g',t')} = 0$ if $t\in V$ and $t'\notin U$. Therefore $F(\phi)$ is continuously controlled along $T \times 1$. 
\qed

\providecommand{\bysame}{\leavevmode\hbox to3em{\hrulefill}\thinspace}
\providecommand{\MR}{\relax\ifhmode\unskip\space\fi MR }
\providecommand{\MRhref}[2]{%
  \href{http://www.ams.org/mathscinet-getitem?mr=#1}{#2}
}
\providecommand{\href}[2]{#2}

\end{document}